\documentclass[11pt, a4paper]{amsart}

\usepackage[T1]{fontenc}
\usepackage{microtype}
\usepackage{amscd}
\usepackage{amsmath}
\usepackage{amsfonts}
\usepackage{amssymb}
\usepackage{bbold}
\usepackage{bm}
\usepackage{hyperref}
\usepackage{xcolor}
\usepackage{enumerate}

\usepackage{tikz,tikz-cd,tikz-3dplot}
\usetikzlibrary{matrix}
\usetikzlibrary{arrows}
\usepackage{subcaption}
\usepackage{wrapfig}
\usepackage{multicol}

\usepackage{algorithmicx}
\usepackage{algorithm}
\usepackage{algpseudocode}

\MakeRobust{\Call}

\algdef{SE}[DOWHILE]{Do}{DoWhile}{\algorithmicdo}[1]{\algorithmicwhile\ #1}
\algrenewcommand{\algorithmiccomment}[1]{\hfill $\rhd$ \emph{#1}}
\algrenewcommand{\algorithmicrequire}{\textbf{Input:}}
\algrenewcommand{\algorithmicensure}{\textbf{Output:}}
\algnewcommand{\Or}{\textbf{or}}
\algnewcommand{\And}{\textbf{and}}
\algnewcommand{\Not}{\textbf{not}\,}
\algnewcommand\algorithmicforeach{\textbf{for each}}
\algdef{S}[FOR]{ForEach}[1]{\algorithmicforeach\ #1\ \algorithmicdo}

\renewcommand{\P}{\mathbb{P}}

\newcommand{\Z}{\mathbb{Z}}

\newcommand{\Q}{\mathbb{Q}}

\makeatletter
\newcommand{\leqnomode}{\tagsleft@true}
\newcommand{\reqnomode}{\tagsleft@false}
\makeatother

\newtheorem*{theorem*}{Theorem}

\newtheorem{theorem}{Theorem}[section]
\newtheorem{lemma}[theorem]{Lemma}
\newtheorem{proposition}[theorem]{Proposition}

\newtheorem{corollary}[theorem]{Corollary} 

\newtheorem{innercustomthm}{Theorem}
\newenvironment{customthm}[1]
{\renewcommand\theinnercustomthm{#1}\innercustomthm}
{\endinnercustomthm}

\theoremstyle{definition}
\newtheorem{definition}[theorem]{Definition}
\newtheorem{example}[theorem]{Example}

\newtheorem{remark}[theorem]{Remark}

\usepackage[left=3.5cm,right=3.5cm,top=3.8cm,bottom=3.8cm]{geometry}

\makeatletter
\def\thm@space@setup{%
  \thm@preskip=5pt \thm@postskip=5pt
}
\makeatother

\title[Graph Curve Matroids]{Graph Curve Matroids}

\author{Alheydis Geiger}
\address{Max-Planck-Institute for Mathematics in the Sciences, Leipzig, Germany}
\email{alheydis.geiger@mis.mpg.de}

\author{Kevin K\"uhn}
\address{Institut f\"ur Mathematik, Goethe--Universit\"at Frankfurt,
60325 Frankfurt am Main, Germany}
\email{kuehn@math.uni-frankfurt.de}

\author{Raluca Vlad}
\address{Department of Mathematics, Brown University, Providence, RI, USA}
\email{raluca\_vlad@brown.edu}

\begin{document}
\begin{abstract}
We introduce a new class of matroids, called \emph{graph curve matroids}. A graph curve matroid is associated to a graph and defined on the vertices of the graph as a ground set. 
We prove that these matroids provide a combinatorial description of hyperplane sections of degenerate canonical curves in algebraic geometry. Our focus lies on graphs that are 
$2$-edge-connected and trivalent. These define \emph{identically self-dual} graph curve matroids, 
but we also develop generalizations. 
Finally, we provide an algorithm to compute the graph curve matroid associated to a given graph, as well as an
implementation in \texttt{Macaulay2} and data of examples.  
\end{abstract}

\maketitle

\section{Introduction}

Connections between matroids and graphs are well-studied in combinatorics and neighboring areas \cite{Oxley2005OnTI,Tutte1959MatroidsAG}. However, the study of matroids associated to curves and their algebraic geometry is less explored. Some advances in this area are \cite{graphcurves_colorings_matroids,geiger2022selfdual}. This paper explores a new direction connecting matroids, graphs, and algebraic curves. 

In \cite{geiger2022selfdual}, matroids are associated to canonical curves by means of a generic hyperplane section and the resulting point configuration. As a special case, \cite{geiger2022selfdual} considers graph curves -- a degeneration of smooth canonical curves into unions of lines. A graph curve of genus $g$, as defined in~\cite{BE91}, is the union of $2g-2$ lines in $\P^{g-1}$ such that each line intersects exactly three others in nodes.
Up to projective transformation, a graph curve is uniquely defined by its dual graph. This dual graph has a vertex for each line of the curve, and an edge between two vertices when
the corresponding lines intersect. The intersection of a graph curve with a generic hyperplane of $\P^{g-1}$ yields a configuration of $2g-2$ points of the hyperplane, which in turn yields a realizable matroid on $2g-2$ elements. This associated matroid only depends on the graph dual to the curve~\cite[Theorem 5.5]{geiger2022selfdual}. Note that this matroid is fundamentally different from the graphic matroid or the cographic matroid associated to a graph, as the ground set of the former is the set of vertices of the graph, while the ground set of the latter is the set of edges.

The focus in \cite{geiger2022selfdual} is on the algebraic description of this matroid associated to a graph via its graph curve. 
In this paper, we provide a combinatorial point of view on this topic.
We exhibit a purely combinatorial description of the matroid, 
which can be computed using Algorithm~\ref{algo:matroid}. Our description and algorithm do not depend on the graph curve itself nor on the choice of the generic hyperplane, but only on the dual graph. More precisely, we introduce a \emph{graph curve matroid} associated to \emph{any} graph (Definition \ref{def:matroid_graphcurve}), and we show that in the case of graphs dual to graph curves, our combinatorial definition coincides with the algebraic construction in \cite{geiger2022selfdual}.

\begin{customthm}{\ref{thm:graphcurve}}
The matroid arising from a generic hyperplane section of a graph curve coincides with the graph curve matroid associated to the dual graph of the curve.
\end{customthm}

Generic hyperplane sections of canonical curves give (identically) self-dual (or self-associated) point configurations as investigated in \cite{CobleAssociatedSO,dolgachev_ortland, Petrakiev}. In the language of matroids, this means that the complement of every basis is again a basis. The property of being identically self-dual (ISD) for matroids has been studied in \cite{Lindstrm, gracia_padro}.
It is a consequence of the Riemann-Roch Theorem for singular curves that the graph curve matroid from a generic hyperplane section of a graph curve is ISD \cite{geiger2022selfdual}. 
We observe that this property and its algebraic proof hold even more generally for the graph curve matroid associated to any trivalent, $2$-edge-connected graph. 
Moreover, using our combinatorial description, we prove this self-duality only from combinatorial principles. 

\begin{customthm}{\ref{thm:self-duality}}
Let $G$ be a trivalent, connected graph. Then its graph curve matroid $M_G$ is identically self-dual if and only if $G$ is 2-edge-connected. 
\end{customthm}

The theory developed in this paper for graph curve matroids is not only helpful in filling the combinatorial gap from \cite{geiger2022selfdual}, but also opens up new interesting questions in combinatorial algebraic geometry. For example, one can ask for an axiomatization of this new class of matroids. 
We know that graph curve matroids are representable over $\Q$ when they come from hyperplane sections of graph curves (Theorem~\ref{thm:graphcurve}), but representability in the general case remains an open question.
One could also investigate further connections between graphs and their graph curve matroids, like conditions for a non-trivalent graph to give an ISD graph curve matroid, or whether there is a combinatorial characterization for when two non-isomorphic graphs define isomorphic graph curve matroids. 
Moreover, one could investigate relations between our graph curve matroids and graph divisors in the chip-firing literature, such as \cite{baker_norine, YUEN2017}.

This paper is structured as follows. Section~\ref{section_matroids} presents the main definition of the graph curve matroid and revisits its connection to graph curves in algebraic geometry. Section~\ref{section_circuits} develops an understanding of the circuits of the graph curve matroid in terms of the underlying graph. Using these results, Section~\ref{section_selfdual} gives a combinatorial proof of the fact that
the graph curve matroids coming from trivalent, 2-edge-connected graphs are identically self-dual. Finally, Section~\ref{section-non-isomorphic-graphs} presents first results on the question 
(first stated in \cite{geiger2022selfdual}) 
of when non-isomorphic graphs define isomorphic graph curve matroids.

The mathematical data for this paper, which includes examples of graphs and their associated graph curve matroids, as well as the implementation of the algorithms in \texttt{Macaulay2}~\cite{M2}, is available at:
\begin{center}
    \url{https://mathrepo.mis.mpg.de/GraphCurveMatroids}. 
\end{center}

\noindent

\subsection*{Funding}
This project has received funding from the Deutsche Forschungsgemeinschaft (DFG, German Research Foundation) TRR 326 \emph{Geometry and Arithmetic of Uniformized Structures}, project number 444845124, from the DFG Sachbeihilfe \emph{From Riemann surfaces to tropical curves (and back again)}, project number 456557832, as well as the DFG Sachbeihilfe \emph{Rethinking tropical linear algebra: Buildings, bimatroids, and applications}, project number 539867663, within the SPP 2458 \emph{Combinatorial Synergies}.  
Part of this project was carried out while the third author visited the Max Planck Institute for Mathematics in the Sciences.

\subsection*{Acknowledgements}
The authors would like to thank 
M. Chan, 
K. Devriendt, 
A. Gross, 
M. Joswig, 
L. Kastner, 
A. Kuhrs,  
B. Schröter, 
B. Sturmfels, and 
M. Ulirsch for helpful discussions and comments.
This project originated at the Graduate Student Meeting on Applied Algebra and Combinatorics at KTH Royal Institute of Technology and Stockholm University 2023.

\section{Matroids and Graphs}\label{section_matroids}

For any matroid $M$ on a finite ground set $E$, its \emph{dual matroid} $M^*$ is the matroid on the same ground set, whose bases are the complements of the bases of $M$. Let $r$ be the rank function of $M$, which assigns to each subset of $E$ the cardinality of any maximal independent subset. Then, by \cite[Proposition~2.1.9]{oxley2011matroid}, we have the following relation for the rank function $r^*$ of the dual matroid $M^*$:
\begin{equation}\label{eqn:rank-fcn-dual-matroid}
r^*(X)=r(E-X)+|X|-r(E) \;\;\;\;\;\; \text{ for all } X \subseteq E.
\end{equation}
A matroid is called \emph{identically self-dual} (ISD) if $r=r^*$.

Throughout this paper, graphs are assumed to be undirected, but need not be simple, i.e., we allow parallel edges and loops. Let $G = (V,E)$ be such a graph. 
For any subset $A \subseteq V$ of vertices or $X \subseteq E$ of edges, 
let $G[A]$ and $G[X]$ denote the subgraphs of $G$ induced on $A$ and $X$, respectively. We denote by $\omega(A)$ and $\omega(X)$ the number of connected components of $G[A]$ and $G[X]$, respectively. Moreover, for any subset $A \subseteq V$ of vertices, we let $\delta(A) \subseteq E$ be the set of edges of $G$ that are incident to at least one vertex in $A$.

By \cite[\S1.3.8]{oxley2011matroid}, the rank function of the \emph{graphic} or \emph{cycle matroid} of $G$ is given by:
\begin{equation}\label{eqn:rank-fcn-graphic-matroid}
r(X) = |V(G[X])| - \omega(X) \;\;\;\;\;\; \text{ for all } X \subseteq E.
\end{equation}
Henceforth, we will use $r$ exclusively for the rank function of a graphic matroid. Moreover, $r^*$ will denote the rank function of its dual matroid, called the \emph{cographic} or \emph{bond matroid} of $G$.

\subsection{The Graph Curve Matroid}
Recall that for any finite ground set $V$, a function $f:2^V\to \Z$ is \emph{submodular} if $f(A \cup B)+f(A\cap B) \leq f(A) + f(B)$ for any $A,B \subseteq V$. Furthermore, $f$ is \emph{increasing} if $A\subseteq B$ implies $f(A)\leq f(B)$. 

\begin{lemma}\label{lemma-submodular}
For any graph $G=(V,E)$, the map
\begin{align*}
f:2^{V}&\longrightarrow \Z \\
A&\longmapsto r^*(\delta(A))-1
\end{align*}
is submodular and increasing. 
\end{lemma}
\begin{proof}For $A,B\subseteq V$, we have $\delta(A \cup B) = \delta(A) \cup \delta(B)$ and $\delta(A \cap B) \subseteq \delta(A) \cap \delta(B)$. Since the rank function $r^*$ 
is submodular and increasing, so is $r^*(\delta(\cdot)) - 1$. 
\end{proof}

If $G$ is simple, then Lemma \ref{lemma-submodular} is a special case of \cite[Example 1.8]{Lovasz}, where we consider the edge set as a system of $2$-element subsets of $V$.

\begin{definition}\label{def:matroid_graphcurve}
Given a graph $G=(V,E)$, its \emph{graph curve matroid} $M_G$ is the matroid on the ground set $V$ given by the following collection of circuits: 
$$\mathcal{C}= \{A \subseteq V \mid A \text{ is minimal 
among non-empty subsets of $V$ s.t.\ } r^*(\delta(A)) \leq |A|\}.$$
\end{definition}

Note that $M_G$ is a well-defined matroid by \cite[Proposition 11.1.1]{oxley2011matroid} applied to the increasing submodular function $f$ from Lemma \ref{lemma-submodular}. 

\begin{example}
Consider the complete graph $K_4$ on $4$ vertices (see Figure \ref{fig:K4-graph}). Applying the formulas \eqref{eqn:rank-fcn-dual-matroid} and \eqref{eqn:rank-fcn-graphic-matroid}, we get that $r^*(\delta(A))=3$ for a vertex subset $A$ with $|A|\geq 2$, while $r^*(\delta(\{v\}))=2$ for any vertex $v$.
Later, in Lemma~\ref{lemma_rank_bondmatroid}, we provide a general formula for $r^*(\delta(A))$ for any connected graph $G$ and vertex subset $A$.

Continuing our example, Definition~\ref{def:matroid_graphcurve} implies that the circuits of $M_{K_4}$ are the vertex subsets of size $3$.
Thus, 
$M_{K_4}$ is the uniform matroid~$U_{2,4}$.
\end{example}

\begin{example}
    Let $G$ be the vertex-labelled \emph{cone graph} from Figure~\ref{fig:cone-graph}. Applying formulas \eqref{eqn:rank-fcn-dual-matroid} and \eqref{eqn:rank-fcn-graphic-matroid}, we see that $r^*(\delta(\{3\})) = 1$, hence $\{3\}$ is a loop in $M_G$. Moreover, we get that $r^*(\delta(A)) = 2$ when the vertex subset $A$ equals $\{1\}$, $\{2\}$, or $\{1,2\}$. Therefore, $\{1,2\}$ is the only other circuit of $M_G$, and thus $M_G \cong U_{0,1} \oplus U_{1,2}$.
\end{example}

\begin{figure}[h]
\begin{subfigure}[c]{.45\textwidth}\centering    
\begin{tikzpicture}[scale = 0.35]
\coordinate (1) at (0,0);
\coordinate (2) at (0,4);
\coordinate (3) at (4,0);
\coordinate (4) at (4,4);

\foreach \i in {1,2,3,4} {\path (\i) node[circle, black, fill, inner sep=2]{};}

\draw[black] (1)--(2)--(3)--(4)--(1)--(3);
\draw[black] (2)--(4);
\node[left] at (1) { 1};
\node[left] at (2) { 2};
\node[right] at (3) { 3};
\node[right] at (4) { 4};

\end{tikzpicture}

\caption{The complete graph $K_4$.} \label{fig:K4-graph}
\end{subfigure}
\begin{subfigure}[c]{.45\textwidth}\centering    
\begin{tikzpicture}[scale = 0.35]
\coordinate (1) at (0,0);
\coordinate (2) at (0,4);
\coordinate (3) at (4,2);

\foreach \i in {1,2,3} {\path (\i) node[circle, black, fill, inner sep=2]{};}

\draw[black] (1)--(3)--(2);
\draw[black] (1) .. controls (0.8,1) and (0.8,2.7) .. (2);
\draw[black] (1) .. controls (-0.8,1) and (-0.8,2.7) .. (2);

\node[left] at (1) {1};
\node[left] at (2) {2};
\node[right] at (3) {3};
\end{tikzpicture}

\caption{The cone graph.} \label{fig:cone-graph}
\end{subfigure} 
\caption{}
\end{figure}

We emphasize that the ground set of a graph curve matroid $M_G$ is the vertex set of the graph $G$, unlike the ground set of the (co)graphic matroid, which is the edge set of $G$.
From now on, $r_{M_G}$ will denote the rank function of $M_G$.
We will justify the term ``graph curve matroid'' in Subsection~\ref{section_graphCurvesCanonicalMaps}, in regard to the algebro-geometric interpretation of this matroid in the case of trivalent, $2$-edge-connected graphs. 
Also, note that a graph curve matroid need not correspond to a unique graph (see Section~\ref{section-non-isomorphic-graphs}). 

Following Definition~\ref{def:matroid_graphcurve}, Algorithm~\ref{algo:matroid} computes the circuits of a graph curve matroid.

\algloopdefx{NoEndIf}[1]{\textbf{if} #1 \textbf{then}}
\algloopdefx{NoEndFor}[1]{\textbf{for each} #1 \textbf{do}}

\begin{algorithm}[h!] \label{alg:matroid}
\caption{Computing the graph curve matroid $M_G$}
\label{algo:matroid}
\begin{algorithmic}[1]
\Require{Graph $G=(V,E)$.}
\Ensure{List $L$ of circuits of $M_G$.}
\State $L:= \emptyset$
\NoEndFor{$C\subseteq V$ with $C \neq \emptyset$} 
\NoEndIf{$r^*(\delta(C))\leq |C|$ \textbf{ and } $r^*(\delta(A))>|A| \; \; \;  \forall \, A \subsetneq C$ with $A \neq \emptyset$}
\State $L = L \cup \{C\}$ 
\State \Return $L$
\end{algorithmic}
\end{algorithm}

The following criteria for (in)dependence in $M_G$ are immediate from Definition~\ref{def:matroid_graphcurve}.

\begin{proposition}\label{prop_dependenceM_G}
A vertex subset $A\subseteq V$ is \dots
\begin{enumerate}[(a)]\item \dots dependent in $M_G$ if and only if it contains a subset $\emptyset \neq A'\subseteq A$ with 
$$r^*(\delta(A'))\leq |A'|.$$
\item \dots independent in $M_G$ if and only if for all subsets $\emptyset\neq A' \subseteq A$ we have
$$r^*(\delta(A'))>|A'|.$$
\end{enumerate}
\end{proposition}

\begin{example}
Checking all non-empty subsets of $A$ in Proposition \ref{prop_dependenceM_G} (b) is necessary, as there can be a vertex subset $A$ that is dependent in $M_G$ but satisfies $r^*(\delta(A))> |A|$. For example, in the \emph{double house graph} in Figure \ref{fig:double-house-graph}, the subset $A = \{1,2,3,6\}$ is dependent in $M_G$, because it contains the dependent subset $A' = \{1,2,3\}$ with $r^*(\delta(A')) = 3 = |A'|$. However, computation shows that $r^*(\delta(A)) = 5 > |A|$.
\end{example}

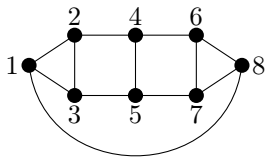
\begin{figure}[h]
    \centering
    \begin{tikzpicture}[scale = 0.4]

\coordinate (1) at (0,0);
\coordinate (2) at (1.5,1);
\coordinate (3) at (1.5,-1);
\coordinate (4) at (3.5,1);
\coordinate (5) at (3.5,-1);
\coordinate (6) at (5.5,1);
\coordinate (7) at (5.5,-1);
\coordinate (8) at (7,0);

\foreach \i in {1,2,3,4,5,6,7,8} {\path (\i) node[circle, black, fill, inner sep=2]{};}

\draw[black] (6)--(7)--(8)--(6);
\draw[black] (6)--(4)--(2)--(1)--(3)--(5)--(7);
\draw[black] (2)--(3);
\draw[black] (4)--(5);
\draw[black] (1) .. controls (0.5,-4) and (6.5,-4) .. (8);

\node[left] at (1) {\small 1};
\node[above] at (2) {\small 2};
\node[below] at (3) {\small 3};
\node[above] at (4) {\small 4};
\node[below] at (5) {\small 5};
\node[above] at (6) {\small 6};
\node[below] at (7) {\small 7};
\node[right] at (8) {\small 8};
\end{tikzpicture}
    \caption{The double house graph.} 
    \label{fig:double-house-graph}
\end{figure}

\noindent

\subsection{Trivalence and \texorpdfstring{$\mathbf{2}$}{2}-Connectivity}

If a graph $G = (V,E)$ is the disjoint union of two graphs $G_1 = (V_1, E_1)$ and $G_2 = (V_2, E_2)$, then the rank functions of the cographic matroids of these graphs satisfy
$$r^*_{G} (X) = r^*_{G_1} (X \cap E_1) + r^*_{G_2} (X \cap E_2) \hspace{9mm} \text{for all } X \subseteq E.$$
Proposition \ref{prop_dependenceM_G}  immediately implies $M_G = M_{G_1} \oplus M_{G_2}$. Therefore, it suffices to restrict our attention to understanding the construction of $M_G$ for connected graphs.

For a positive integer $k$, a connected graph $G$ is \emph{$k$-edge-connected} if removing any $k-1$ or fewer edges from $G$ does not disconnect the graph. Similarly, $G$ is \emph{$k$-vertex-connected} if removing any $k-1$ or fewer vertices from $G$ does not disconnect it. Whenever we refer to a graph as ``$k$-connected,'' we mean ``$k$-edge-connected.'' We call a graph \emph{trivalent} if every vertex has degree $3$. For trivalent graphs, vertex and edge connectivity are closely related:

\begin{lemma}\label{lemma_k-connectedness}
Let $G$ be trivalent. Then $G$ is \dots
\begin{enumerate}[(a)]
\item \dots $2$-connected if and only if it is loopless and $2$-vertex-connected. 
\item \dots $3$-connected if and only if either it is simple and $3$-vertex-connected or it is the theta graph (cf.\ Figure~\ref{fig:theta}).
\end{enumerate}
\end{lemma}
\begin{proof}
For part (a), trivalence implies that a disconnecting edge must be incident to a loop or a disconnecting vertex, and vice versa.
For (b), see \cite[Lemma 2.6]{BE91}.
\end{proof}

Even though our graph curve matroid is well-defined for any graph, for the rest of the paper we will mostly restrict our attention to trivalent, $2$-connected graphs, since this case results in interesting matroid properties and algebro-geometric interpretations. Below, we state some properties of $M_G$ when $G$ is trivalent, related to $2$-connectivity.

\begin{lemma}\label{lemma_equality_circuits}
Let $G = (V,E)$ be trivalent.
\begin{enumerate}[(a)]
\item The graph curve matroid $M_G$ is loopless if and only if $G$ is $2$-connected.
\item If $G$ is $2$-connected, then any circuit $C\subseteq V$ in $M_G$ satisfies $r^*(\delta(C))=|C|$.
\end{enumerate}
\end{lemma}
\begin{proof}
(a) Let $v \in V$ be a vertex of $G$. Using \eqref{eqn:rank-fcn-dual-matroid} and \eqref{eqn:rank-fcn-graphic-matroid}, we get:
$$r^*(\delta(\{v\}))=r(E-\delta(\{v\}))+|\delta(\{v\})|-r(E)=r(E-\delta(\{v\}))+|\delta(\{v\})|-(|V|-1).$$
If $G$ is $2$-connected, then Lemma~\ref{lemma_k-connectedness} implies that $v$ is not disconnecting or incident to a loop, so $r(E-\delta(\{v\})) = |V|-2$ and $|\delta(\{v\})| = 3.$ Thus, $r^*(\delta(\{v\}))= 2$, so $\{v\}$ is not a circuit in $M_G$. On the other hand, if $v$ is disconnecting or incident to a loop, computation shows that $r^*(\delta(\{v\})) \in \{0,1\}$, so $\{v\}$ forms a loop in $M_G$.

\medskip

(b) By part (a), any circuit $C$ has at least two elements. Thus, $C-\{v\}$ is non-empty and independent for any $v \in C$. By Proposition~\ref{prop_dependenceM_G} and the fact that $r^*$ is increasing, we get the following chain of inequalities:
$$|C|-1<r^*(\delta(C-\{v\})) \leq r^*(\delta(C)) \leq |C|.$$
\end{proof}

\subsection{Graph Curves and Canonical Morphisms}\label{section_graphCurvesCanonicalMaps}

We recall a few terms and basic results about graph curves, for more see \cite{BE91,BRSV15}. 
For us, a \emph{graph curve} $C$ is a connected, projective algebraic curve whose components are copies of the projective line, such that each component intersects the others in precisely three nodes.

The isomorphism type of a graph curve is determined by the isomorphism type of its \emph{dual graph}, which has a vertex for each component of the curve and an edge for every intersection point of two components. Vice versa, for every trivalent, loopless, connected graph $G$, one can construct a graph curve with dual graph $G$. 

\begin{remark}
Graph curves are maximally degenerate stable curves. Up to isomorphism, each graph curve corresponds to a unique point in a deepest boundary stratum of the Deligne--Mumford compactification of the moduli space of smooth curves with a fixed genus~\cite{Chan2021ModuliSO, DM69}.
\end{remark}

From now on, we fix a graph curve $C$ with (trivalent) dual graph $G$. Let $\omega_C$ be its \emph{dualizing sheaf} (see \cite[Section~III.7]{Hartshorne}). It is the line bundle over $C$ constructed in the following way. 
Let $C'= \sqcup_{v \in V(G)} \P^1$ be the \emph{normalization} of $C$.
The smooth curve $C'$ is obtained by resolving the (nodal) singularities of $C$ and comes with a natural map $C' \to C$ such that the preimage of a point $p \in C$ consists of two points when $p$ is a node, and one point otherwise.
Then, $\omega_C$ is the push-forward along this map of the sheaf of rational differential forms on $C'$ with at worst simple poles at the points lying over the nodes of $C$ such that the residues at any two points lying over the same node add up to zero. 
In particular, on each component of $C$, $\omega_C$ restricts to $\mathcal{O}_{\P^1}(1)$, which is obtained from the sheaf of regular differentials $\omega_{\P^1} \cong \mathcal{O}_{\P^1}(-2)$ via twisting by the three nodal points on that component.

The \emph{arithmetic genus} of $C$ is given by $g:=
b_1(G)$, i.e., the first Betti number of $G$. Since $G$ is trivalent, we always have $|V|=2g-2$ and $|E|=3g-3$.

By \cite[Proposition~2.5]{BE91}, the dualizing sheaf $\omega_C$ is \emph{basepoint-free} (see \cite[Section~III.7]{Hartshorne}) 
if and only if $G$ is $2$-connected.
In that case, the line bundle $\omega_C$ yields the \emph{canonical morphism}
\begin{align*}
C&\longrightarrow \P(H^0(C,\omega_C)^\vee) \\
c &\longmapsto \big[ s\mapsto s(c) \big]\, .
\end{align*}
In explicit coordinates: any basis $s_1,\dots,s_g$ of $H^0(C,\omega_C)$ yields an identification $\P(H^0(C,\omega_C)^\vee) \cong \P^{g-1}$ under which the canonical morphism is given by
\begin{align*}
C &\longrightarrow \P^{g-1} \\
c &\longmapsto \big( s_1(c):\dots:s_g(c) \big)\, .
\end{align*}
For more details, see \cite[Section~IV.5]{Hartshorne}. Note that this map is well-defined because $\omega_C$ is basepoint-free. Also, this map is unique up to projective linear transformation of $\P^{g-1}$, i.e., up to change of basis $s_1, \dots,s_g$. 
Under this canonical morphism, $\omega_C$ is the pull-back of $\mathcal{O}_{\P^{g-1}}(1)$. Since $\omega_C$ restricts to $\mathcal{O}_{\P^1}(1)$ on each component of $C$, the canonical morphism embeds each component of $C$ onto a line in $\P^{g-1}$.

\begin{remark}\label{rmk:realization-cographic-matroid}
    The images of the nodes of $C$ under this canonical morphism realize the cographic matroid of $G$ \cite[page 3]{BE91}.
\end{remark}

\begin{remark}
The dualizing sheaf is very ample if and only if $G$ is $3$-connected and not the \emph{theta graph} (Figure~\ref{fig:theta}) \cite[Proposition 2.5]{BE91}. In that case, the canonical morphism is an embedding, hence $C$ can be identified with its image as an honest arrangement of lines in projective space $\P^{g-1}$. In particular, our notion of graph curve agrees with that of \cite{geiger2022selfdual} in the case of $3$-connected, simple graphs. However, if the dual graph is not $3$-connected and simple, then $C$ cannot be recovered from its image, as can be seen in Examples \ref{ex_theta} and \ref{ex_soda}.
\end{remark}

One realization of the canonical morphism
can be obtained from the combinatorics of $G$, as explained in~\cite[Section~5]{geiger2022selfdual}: Choose $g$ cycles generating the cycle space of $G$ and set $\text{Cyc}_G$ to be the $g\times (3g-3)$ matrix with these generators as edge-incidence row vectors. 
For every vertex $v$ of $G$, 
there is a linearly dependent triple of columns in $\text{Cyc}_G$ indexed by the edges incident to $v$,
which yields three distinct collinear points in $\P^{g-1}$. 
Under the canonical morphism, the component of $C$ corresponding to $v$ maps to the unique line through these three points in $\P^{g-1}$.

\begin{remark}
    By \cite[Theorem~7.47]{aigner_comb_theory}, $\text{Cyc}_G$ realizes the cographic matroid of $G$. This statement is equivalent to Remark~\ref{rmk:realization-cographic-matroid}, because the column vectors of $\text{Cyc}_G$ become the coordinates of the nodes of $C$ under this realization of the canonical map.
\end{remark}

We can associate a matroid to a graph curve using the canonical morphism. 
For a generic hyperplane $\P^{g-2}\cong H \subseteq \P^{g-1}$, the intersection of $H$ with the image of $C$ under the canonical morphism yields a multiset of $2g-2$ points in $H$, one point for every component of $C$. A point could appear with multiplicity $>1$ since the canonical morphism might send distinct components of $C$ to the same line in $\P^{g-1}$; see Figure~\ref{fig:graph_curve} for an example.
As $H$ is generic, the matroid given by the linear dependencies of this point configuration only depends on $G$ and will be denoted $M_G'$.

This process generalizes the one from \cite{geiger2022selfdual}, where only graph curves with $3$-connected, simple dual graphs are considered. In that case, $C$ is identified with its image under the canonical morphism, and all intersection points with a generic $H$ are distinct and have multiplicity~$1$. 

\begin{example}\label{ex_theta}
Let $C$ be the graph curve corresponding to the theta graph in Figure~\ref{fig:theta}. That is, $C$ is the union of two (abstract) projective lines that intersect transversally in three points. 
Hence, the canonical morphism $C\to \P^{g-1}=\P^1$ restricts to the identity on each line. A generic hyperplane of $\P^1$ is just a generic point of $\P^1$. This yields a point configuration consisting of a double point, whose matroid is $U_{1,2}$.
\end{example}

\begin{figure}[h]
\centering
\begin{subfigure}[c]{.45\textwidth}\centering    
\begin{subfigure}[c]{.8\textwidth}\centering
\begin{tikzpicture}[scale = 0.6]
\coordinate (1) at (0,0);
\coordinate (2) at (2,0);
\foreach \i in {1,2} {
\path (\i) node[circle, black, fill, inner sep=2]{};
}
\draw[black] (1)--(2); 
\draw[black] (1) to[out=70,in=110,distance=1.1cm]  (2);
\draw[black] (1) to[out=-70,in=-110,distance=1.1cm]  (2);

\node[left] at (1) {1};
\node[right] at (2) {2};
\end{tikzpicture}
\caption{The theta graph.}\label{fig:theta}
\end{subfigure}
\quad
\begin{subfigure}[c]{0.8\textwidth}\centering
\begin{tikzpicture}[scale = 0.6]
\coordinate (3) at (0,0);
\coordinate (1) at (0,2);
\coordinate (4) at (2,0);
\coordinate (2) at (2,2);
\foreach \i in {1,2,3,4} {
\path (\i) node[circle, black, fill, inner sep=2]{};
}
\draw[black] (1)--(3)  (4) -- (2) ; 
\draw[black] (1) to[out=40,in=140,distance=0.8cm]  (2);
\draw[black] (1) to[out=-40,in=-140,distance=0.8cm]  (2);

\draw[black] (3) to[out=40,in=140,distance=0.8cm]  (4);
\draw[black] (3) to[out=-40,in=-140,distance=0.8cm]  (4);

\node[left] at (1) {1};
\node[right] at (2) {2};
\node[left] at (3) {3};
\node[right] at (4) {4};
\node[right] at (0,4) {};

\end{tikzpicture}
\caption{The soda can graph.}\label{fig:soda_can}
\end{subfigure}
\end{subfigure}
\begin{subfigure}[c]{.45\textwidth}\centering
\begin{tikzpicture}[scale = 0.6]
\coordinate (1) at (0,2);
\coordinate (2) at (6,2);
\coordinate (3) at (0,0);
\coordinate (4) at (6,0);
\coordinate (5) at (2,3);
\coordinate (6) at (3.1,1.5);
\coordinate (7) at (2.7,0.2);
\coordinate (8) at (4,-1);
\coordinate (9) at (2.9,0.5);
\coordinate (10) at (3.3,1.8);

\draw[black] (1) to[out=-70,in=-110,distance=3cm]  (2);
\draw[teal] (3) to[out=70,in=110,distance=3cm]  (4);
\draw [blue] plot [smooth, tension=1] coordinates { (5) (6) (7) };
\draw [purple] plot [smooth, tension=1] coordinates { (8) (9) (10) };
\node[left] at (1) {1};
\node[right, teal] at (4) {2};
\node[left, blue] at (5) {4};
\node[right, purple] at (8) {3};

\draw[->]{(3,-1) -- (3,-2)};
\draw[blue]{(2.95,-3) -- (2.95,-6)};
\draw[purple]{(3.05,-3) -- (3.05,-6)};
\draw[black]{(0,-4.5) -- (6,-4.5)};
\draw[teal]{(0,-4.6) -- (6,-4.6)};

\end{tikzpicture}
\caption{The graph curve dual to the soda can graph, with its canonical morphism to $\P^2$.
}\label{fig:graph_curve}
\end{subfigure}
\caption{Two non-simple graphs and a canonical morphism.}
\label{fig:two_connected_graphs}
\end{figure}
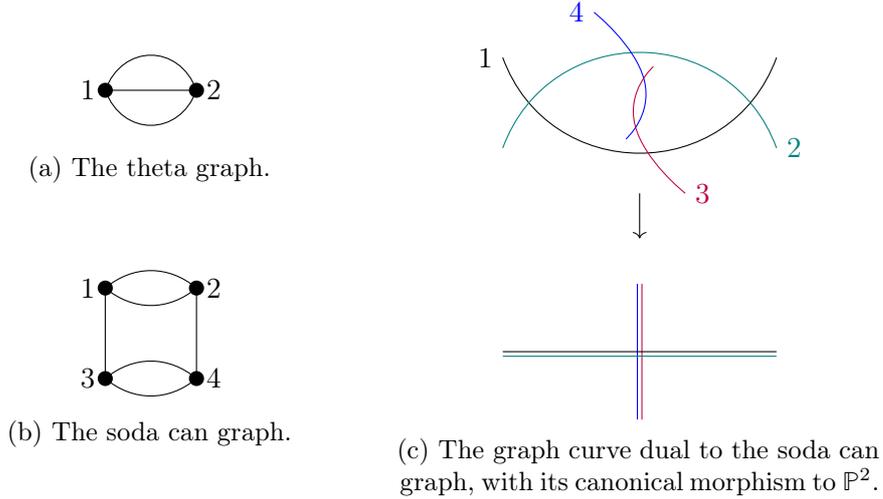

\begin{theorem}\label{thm:graphcurve}
For a trivalent, $2$-connected graph $G$, the graph curve matroid $M_G$ coincides with the matroid $M_G'$ arising from a generic hyperplane section of the graph curve dual to $G$.

\end{theorem}
\begin{proof} Let $G=(V,E)$ be trivalent and 2-connected with dual graph curve $C$.
For every vertex $i \in V$, let $L_i \subseteq \P^{g-1}$ denote the line that is the image of the corresponding component of $C$ under the canonical morphism. 
Every edge $e = (i,j) \in E$ corresponds to a point $q_e \in  L_i\cap L_j \subseteq \P^{g-1}$. More precisely, the edge $e$ in $G$ corresponds to an intersection point of the components $i$ and $j$ of $C$, and $q_e$ is the image of this node under the canonical morphism.

Let $\emptyset \neq A\subseteq V$. By \cite[Lemma 5.7]{geiger2022selfdual}, $A$ is a circuit in $M_G'$ if and only if $A$ is inclusion-minimal with the property that the dimension of $\text{span}(\{ L_i \mid i \in A\}) = \text{span}(\{q_e \mid e \in \delta(A)\})$ equals $|A|-1$.
The points $q_e$ realize the cographic matroid, so
$$\dim \, \text{span}(\{q_e \mid e \in \delta(A)\})=r^*(\delta(A))-1\, .$$
Thus, $A$ is a circuit in $M_G'$ if and only if $A$ is inclusion-minimal with $r^*(\delta(A))-1=|A|-1$, that is, if and only if $A$ is a circuit in $M_G$ by Proposition~\ref{prop_dependenceM_G} and Lemma~\ref{lemma_equality_circuits}.
\end{proof}

The following theorem is a slightly stronger version of \cite[Theorem 5.5]{geiger2022selfdual}. All arguments go through, as the main tools are the Riemann-Roch Theorem and the fact that $\omega_C$ is the pull-back of $\mathcal{O}_{\P^{g-1}}(1)$, which still hold with our weaker assumptions on the graph $G$.

\begin{theorem}\label{thm_geometric_self-duality}
For a $2$-connected, trivalent graph $G$, the matroid $M_G$ is ISD. 
\end{theorem}

In Section \ref{section_selfdual}, we will give a purely combinatorial proof that $M_G$ is identically self-dual.


\begin{example}\label{ex_soda}
Consider the \emph{soda can graph} in Figure \ref{fig:soda_can}. The corresponding graph curve $C$ has four components intersecting as in Figure \ref{fig:graph_curve}. Each pair of (abstract) lines in $C$ that intersect in two points must be mapped to the same line in $\P^{g-1}=\P^2$ by the canonical morphism. Hence, the image of $C$ under the canonical morphism consists of two (double) lines intersecting transversely. Slicing with a generic hyperplane in $\P^2$ yields two double points. Thus, the graph curve matroid equals $U_{1,2}\oplus U_{1,2}$, which is identically self-dual.
\end{example}

\section{Circuits in \texorpdfstring{$M_G$}{MG}}\label{section_circuits}
Throughout this section, we fix a trivalent, $2$-connected graph $G=(V,E)$. Note that by Lemma \ref{lemma_k-connectedness}, $G$ cannot have loops, but we allow $G$ to have parallel edges. 
We say that $C = \{v_1,\dots,v_r\} \subseteq V$ is a \emph{cycle} if, after relabeling the vertices if necessary, there is an edge between $v_i$ and $v_{i+1}$ for every $1 \leq i \leq r$, where $v_{r+1} = v_1$. If $r=2$, we require these edges to be distinct, i.e., there is a pair of parallels between $v_1$ and $v_2$. Finally,  a vertex subset of a graph is \emph{cyclic} if the induced subgraph contains at least one cycle, and \emph{acyclic} otherwise.

This section provides a complete combinatorial description of the circuits of $M_G$ from the graph $G$, as given in Proposition~\ref{prop_circuitsInM_G}. This description will be very useful for the main result of Section~\ref{section_selfdual}. Recall that for a subset $A\subseteq V$ of vertices, $\omega(A)$ refers to the number of components of the induced subgraph.

\begin{proposition}\label{prop_circuitsInM_G}
A non-empty subset $A \subseteq V$ is a circuit in $M_G$ if and only if either:
\begin{itemize}
\item $A$ is a cycle, or 
\item $A$ is acyclic and $\omega(A) +1 = \omega(V-A)$,
\end{itemize}
and $A$ contains no proper non-empty subset for which either property holds.
\end{proposition}
To prove this, we first need to understand how to compute $r^*(\delta(A))$ for any $A\subseteq V$.
\begin{lemma}\label{lemma_rank_bondmatroid}
For every $A\subseteq V$, we have 
$$r^*(\delta(A))= |\delta(A)|-|A|-\omega(V-A)+1.$$
In particular, if $A$ is acyclic, then 
$$r^*(\delta(A))= |A| + \omega(A)-\omega(V-A)+1.$$
\end{lemma}
\begin{proof}
Applying \eqref{eqn:rank-fcn-graphic-matroid}, we get that $r(E) = |V|-1$ and 
$$r(E - \delta(A)) = |V(G[E - \delta(A)])| - \omega(E - \delta(A)) = |V(G[V-A])| - \omega(V - A),$$
where the last equality holds because $G[V-A]$ equals $G[E - \delta(A)]$ together with an isolated node for each $v \in V- A$ which, in $G$, neighbors only vertices in $A$. The first statement of the lemma then follows from \eqref{eqn:rank-fcn-dual-matroid}.
If $A$ is acyclic, then trivalence implies $|\delta(A)|=2\cdot|A|+\omega(A)$.
\end{proof}

We remark that the first equation in Lemma \ref{lemma_rank_bondmatroid} also holds for any connected graph, not just for trivalent, $2$-connected ones. Applying Lemmas \ref{lemma_equality_circuits} and \ref{lemma_rank_bondmatroid}, the following is a restatement of Proposition \ref{prop_dependenceM_G}: 

\begin{proposition}\label{prop_circuitComponents}
Let $A\subseteq V$ be acyclic. Then $A$ is \dots
\begin{enumerate}[(a)]
\item \dots dependent if and only if it contains a subset $\emptyset\neq A'\subseteq A$ with
$$\omega(A') < \omega(V-A').$$ 
Moreover, if $A$ is a circuit, then 
$$\omega(A) +1 = \omega(V-A)\, .$$
\item \dots independent if and only if for all subsets $\emptyset \neq A'\subseteq A$ we have
$$\omega(A')\geq \omega(V-A')\, .$$
\end{enumerate}
\end{proposition}

\begin{example}
Consider the double house graph $G$ in Figure \ref{fig:double-house-graph}. The vertex subset $A = \{2,3,5,6,7\}$ is dependent in $M_G$ since it is acyclic, connected, and disconnects $G$. 
Moreover, every proper subset $A' \subsetneq A$ is independent because $\omega(A') \geq \omega(V-A')$, so $A$ is a circuit.
We also note that circuits of a graph curve matroid need not be connected. For example, $B = \{1,2,5,6,8\}$ is a circuit in $M_G$ and $G[B]$ is disconnected. 
\end{example}

\begin{lemma}\label{lemma:cyclic-dependent}
Any cyclic subset is dependent in $M_G$.
\end{lemma}
\begin{proof}
Let $A\subset V$ be an inclusion-minimal cycle. If $G$ is the theta graph (Figure \ref{fig:theta}), the statement is clear. Otherwise, $A\subsetneq V$ and $|\delta(A)|= 2\cdot|A|$. By Lemma \ref{lemma_rank_bondmatroid}, we get 
$$r^*(\delta(A))=2\cdot|A|-|A|-\omega(V-A)+1\leq |A|,$$ 
so $A$ is dependent. Thus, any cyclic subset of $V$ must also be dependent.
\end{proof}

Proposition \ref{prop_circuitsInM_G} now follows immediately from Proposition \ref{prop_circuitComponents} and Lemma \ref{lemma:cyclic-dependent}. 

By Lemma~\ref{lemma:cyclic-dependent}, if a cyclic subset $A \subseteq V$ gives a circuit in $M_G$, then $A$ must necessarily be an inclusion-minimal cycle. The following result provides a sufficient condition. 
\begin{proposition}
An inclusion-minimal cycle $A \subseteq V$ is a circuit in $M_G$ if and only if $A$ does not disconnect the graph $G$.
\end{proposition}
\begin{proof}
Let us assume $G$ is not the theta graph (Figure \ref{fig:theta}), since the unique cycle of the theta graph does not disconnect it.

Let $A$ be a circuit in $M_G$, $v \in A$, and $A' = A - \{v\}$. Since $A'$ is acyclic, connected, and independent in $M_G$, applying Proposition~\ref{prop_circuitComponents} (b) yields
$$1 = \omega(A') \geq \omega(V - A').$$
Thus, $G[V - A']$ is connected. Since the vertex $v$ has degree $1$ in $G[V - A']$, the graph $G[V-A]$ must also be connected. 

Conversely, assume that $G[V - A]$ is connected. Every non-empty subset $A' \subsetneq A$ is acyclic. Moreover, $G[V-A']$ is connected because each vertex in $A -A'$ is adjacent to $G[V - A]$. Hence, we have that $\omega(A') \geq \omega(V - A') = 1,$ so $A'$ is independent in $M_G$ by Proposition~\ref{prop_circuitComponents}.
\end{proof}

\section{Self-Duality}\label{section_selfdual}
In \cite{geiger2022selfdual}, the ISD property of matroids coming from graph curves was argued via the Riemann-Roch Theorem. Here, we present a combinatorial proof. 

In this section, $G=(V,E)$ always denotes a trivalent, $2$-connected graph.
For us, the \emph{genus} of a graph $G$ is equal to its first Betti number, i.e., the rank of its cycle space. 
The graph curve matroid $M_G$ of a graph $G$ of genus $g$
has a ground set of cardinality $|V|=2g-2$.
To prove that $M_G$ is ISD, it therefore suffices to show that $M_G$ has rank $g-1$ and that the complement of any basis is independent (hence also a basis).

To accomplish the first step, we construct a set of rank $g-1$ in $M_G$ by starting with any vertex of $G$ and inductively adding adjacent vertices such that the rank in $M_G$ increases sufficiently often. Crucial for this is the following lemma:

\begin{lemma}\label{lemma_rank_increase}
If $A\subseteq V$ and $v\in V-A$ is a vertex contained in a cycle of $G[V-A]$, then $r_{M_G}(A\cup \{v\})>r_{M_G}(A)$. In particular, if $A$ is a basis of $M_G$, then $V-A$ must be acyclic.
\end{lemma}
\begin{proof}
Let $I\subseteq A$ be a maximal independent subset of $A$ in $M_G$. We claim that $I\cup \{v\}$ is also independent. Surely, $I\cup \{v\}$ is acyclic, because $v$ is trivalent and at least two of its incident edges go to vertices of $V-A \subseteq V-I$. 
By Proposition~\ref{prop_circuitComponents}~(b) and the fact that $M_G$ is loopless, it suffices to show that $$\omega(I' \cup \{v\}) \geq \omega(V-(I'\cup \{v\}))$$ 
for all non-empty subsets $I' \subseteq I$.
As $I$ is independent in $M_G$, we have that $\omega(I') \geq \omega(V-I')$.
There are two cases: if $v$ is adjacent to $I'$, then $\omega(I'\cup \{v\})= \omega(I')$ and $\omega(V-(I'\cup \{v\}))=\omega(V-I')$, because $v$ is contained in a cycle and is bivalent in $G[V-I']$. If $v$ is not adjacent to $I'$, then
$\omega(I'\cup \{v\})=\omega(I')+1$ and $\omega(V-(I'\cup \{v\})) \leq \omega(V-I')+1$.
In both cases, the desired inequality holds.
\end{proof}

\begin{lemma}\label{lemma:rank-increse-one-adjacent-vertex}
Let $A \subseteq V$ and $v \in V - A$ a vertex adjacent to $A$.
\begin{enumerate}[(a)]
\item We have $r^*(\delta(A\cup \{v\})) \leq r^*(\delta(A))+1 $.
\item If $r^*(\delta(A\cup \{v\}))=r^*(\delta(A))+1$, then $r_{M_G}(A\cup\{v\}) > r_{M_G}(A).$
\end{enumerate}
\end{lemma}
\begin{proof}

(a) We first show that if $|\delta(A)\cap\delta(v)|\in \{2,3\}$, then $r^*(\delta(A\cup \{v\}))=r^*(\delta(A))$. If $v$ has two edges incident to $A$, then $\omega(V-(A\cup \{v\}))=\omega(V-A)$. 
If all three edges of $v$ are incident to vertices in $A$, then 
$\omega(V-(A\cup \{v\}))=\omega(V-A)-1$, since $v$ is isolated in $G[V-A]$. In both cases, Lemma~\ref{lemma_rank_bondmatroid} tells us that $r^*(\delta(A\cup \{v\}))= r^*(\delta(A)).$

Then, we consider the case when exactly one edge of $v$ is incident to $A$. In this case, 
\reqnomode
\begin{align}\label{eq_1}
\omega(V-(A\cup \{v\})) \geq \omega(V-A).
\end{align}
Thus, by Lemma~\ref{lemma_rank_bondmatroid}, we obtain
\begin{align}\label{eq_2}
r^*(\delta(A\cup \{v\})) &= |\delta(A\cup \{v\})| -|A\cup \{v\}|-\omega(V-(A\cup \{v\}))+1 \nonumber\\
&\leq |\delta(A)|+2-|A|-1-\omega(V-A)+1\\
&= r^*(\delta(A))+1.\nonumber
\end{align}

(b) By Lemma~\ref{lemma_rank_increase}, it suffices to show that $v$ is contained in a cycle of $G[V-A]$. 
Equality in part (a) can only hold if we have equality in \eqref{eq_2}, hence in \eqref{eq_1}.
Let $w_1,w_2\in V-A$ be the unique (not necessarily distinct) vertices adjacent to $v$. As $\omega(V-(A\cup\{v\})=\omega(V-A)$ and $w_1,w_2$ are in the same component in $G[V-A]$ connected by $v$, they must also be in the same component in $G[V-(A\cup \{v\})]$. Hence, there is a simple path through vertices in $V-(A\cup\{v\})$ from $w_1$ to $w_2$. The union of this path with $v$ forms a cycle in $G[V-A]$.
\end{proof}

\begin{proposition}\label{prop_cardinality_basis}
The rank of $M_G$ is $g-1$.
\end{proposition}
\begin{proof}
As $G$ is connected and trivalent, the rank formula for dual matroids \eqref{eqn:rank-fcn-dual-matroid} implies
$$r^*(E)=|E|-r(E) = |E| - (|V| - 1) =(3g-3)-(2g-3)=g.$$ 
By Proposition~\ref{prop_dependenceM_G}, any subset of $V$ of cardinality at least $g$ needs to be dependent in $M_G$. Thus, $r_{M_G}(V) \leq g-1$.
We now show that $r_{M_G}(V)\geq g-1$.
Choose any $v_1 \in V$ and set $A_1=\{v_1\}$. Then, $r^*(\delta(A_1))=2$ by Lemma \ref{lemma_rank_bondmatroid}. By sequentially adding adjacent vertices, construct a chain of vertex subsets
$$\{v_1\}=A_1\subsetneq A_2 \subsetneq \ldots \subsetneq A_{2g-2}=V$$
such that $|A_i|=i$ and $G[A_i]$ is connected. We have that $r^*(\delta(A_1))=2$, while $r^*(\delta(A_{2g-2}))= r^*(E)=g$. Since $r^*(\delta(A_{i+1}))\leq r^*(\delta(A_i))+1$ by Lemma~\ref{lemma:rank-increse-one-adjacent-vertex} (a), there have to be exactly $g-2$ indices $i_1< i_2 <\ldots <i_{g-2}$ with $r^*(A_{i_j+1})= r^*(A_{i_j})+1$. By Lemma \ref{lemma:rank-increse-one-adjacent-vertex} (b), it follows that $r_{M_G}(A_{i_j}) < r_{M_G}(A_{i_j+1})$. Thus, we have
$$1\leq r_{M_G}(A_{i_1}) < r_{M_G}(A_{i_2})<\dots<r_{M_G}(A_{i_{g-2}})<r_{M_G}(A_{i_{g-2}+1})$$
and therefore $r_{M_G}(V) \geq r_{M_G}(A_{i_{g-2}+1})\geq g-1$.
\end{proof}
The proof of Proposition~\ref{prop_cardinality_basis} is constructive, so it yields an algorithm that starts with any vertex of $G$ and constructs a basis of $M_G$ containing the vertex. See Algorithm~\ref{alg:basis}.

\begin{algorithm}[h] 
\caption{Computing a basis of $M_G$ containing a given vertex} \label{alg:basis}
\begin{algorithmic}[1]
\Require{Trivalent, $2$-connected graph $G=(V,E)$, $v\in V$.}
\Ensure{Basis $B$ of $M_G$ containing $v$.}
\State $B:= \{v\}$, $A:=B$
\While{$|B|<g-1$}
\State choose a vertex $w\in V\setminus A$ that is adjacent to $A$
\If{$r^*(\delta(A\cup\{w\}))>r^*(\delta(A))$}
\State $B=B\cup \{w\}$
\EndIf
\State $A=A\cup \{w\}$
\EndWhile
\State \Return $B$
	\end{algorithmic}
\end{algorithm}

\begin{lemma}\label{lemma_components_basis}
Let $B$ be a basis of $M_G$. Then $\omega(V-B)=\omega(B)$.
\end{lemma}
\begin{proof}
Applying the acyclic case of Lemma \ref{lemma_rank_bondmatroid} to $r^*(\delta(B))\leq g$, we obtain that $\omega(B)\leq\omega(V-B)$.
Since $B$ is independent, the claim follows from Proposition~\ref{prop_circuitComponents}~(b).
\end{proof}

\begin{theorem}\label{thm:self-duality}
Let $G$ be a connected, trivalent graph. Then its graph curve matroid $M_G$ is identically self-dual if and only if $G$ is 2-edge-connected. 
\end{theorem}
\begin{proof}
If $G$ contains a disconnecting edge (i.e., a bridge), then $M_G$ has a loop by Lemma~\ref{lemma_equality_circuits} and is therefore not ISD. 

Assume that $G$ is $2$-connected. Let $B$ be a basis of $M_G$. Since $|V-B|=g-1=r_{M_G}(V)$, it suffices to prove that $V- B$ is independent. By Lemma \ref{lemma_rank_increase}, $V-B$ is acyclic. So, by Proposition \ref{prop_circuitComponents} (b), we need to show that for any subset $\emptyset\neq A\subseteq V-B$ we have $\omega(A)\geq \omega(V-A)$. We prove this by contradiction.
We assume that there exists a subset $\emptyset \neq A\subseteq V-B$ with $\omega(A)< \omega(V-A)$.

\emph{Claim:} There exists $A'\subseteq V-B$ containing $A$ such that $\omega(A') < \omega(V-A')$ and such that $G[A']$ is a union of connected components of $G[V-B]$.

To construct $A'$, we inductively add adjacent vertices to $A$. Let $v\in V-(B\cup A)$ be a vertex that is adjacent to $A$. (If no such vertex exists, then we set $A'=A$ and we are done.) 
By a similar analysis as in the proof of Lemma~\ref{lemma:rank-increse-one-adjacent-vertex}~(a), we have $\omega(A\cup \{v\}) < \omega(V-(A\cup \{v\}))$.
Replacing $A$ by $A\cup \{v\}$ and repeating the above procedure until no longer possible yields the desired set $A'$.
Note that $A'\subsetneq V- B$, because otherwise we have 
$$\omega(V-B) = \omega(A') < \omega(V - A') = \omega(B),$$
which contradicts Lemma \ref{lemma_components_basis}.
This proves the claim.

Let now $m=\omega(B)=\omega(V-B)$. We denote by $B_1,\dots,B_m$ the vertex sets of the connected components of $G[B]$ and by $A_1,\dots,A_m$ the vertex sets of the components of $G[V-B]$. 
Since $G[A']\neq G[V-B]$ is a union of components of $G[V-B]$,  the vertex set of any component of $G[V-A']$ is a disjoint union of some $A_i$'s and $B_j$'s. 

For each component $G[C]$ of $G[V-A']$, let $f_A(C):=|\{i \in [m] \mid A_i \subseteq C\}|$ and let $f_B(C) := |\{i \in [m] \mid B_i \subseteq C\}|$. 
There must exist a component $G[C]$ such that $f_A(C)\geq f_B(C)$. Otherwise, $f_B(C)\geq f_A(C)+1$ for all components $G[C]$ of $G[V-A']$, and thus
\begin{equation*}
m = \sum_{\substack{G[C] \textrm{ comp. } \\ \textrm{ of } G[V-A']}} \hspace{-0.2cm}f_B(C) 
\ \geq \ \omega(V-A') +\hspace{-0.2cm}\sum_{{\substack{G[C] \textrm{ comp. } \\ \textrm{ of } G[V-A']}}} \hspace{-0.2cm}f_A(C) 
\ >\  \omega(A') +\hspace{-0.2cm} \sum_{{\substack{G[C] \textrm{ comp. } \\ \textrm{ of } G[V-A']}}} \hspace{-0.2cm}f_A(C)
\ =\  m.
\end{equation*}
Thus, we can now fix a component $G[C]$ of $G[V-A']$ with $f_A(C)\geq f_B(C)$. After relabeling if necessary, we can assume that $C=A_1\cup \ldots \cup A_k\cup B_1\cup \ldots\cup B_l$ with $k\geq l$. 
Then, consider the set $B':=B_1\cup\ldots\cup B_l$, which satisfies
$$\omega(B')=l<k+1\leq \omega(V-B'),$$
where the last inequality follows because $G[A_1], \dots, G[A_k]$ are connected components of $G[V-B']$, and there is at least one other component containing vertices in $A'$.
Thus, by Proposition~\ref{prop_circuitComponents}~(a), $B'$ is dependent in $M_G$.
As $B'\subseteq B$ and $B$ is a basis, this is a contradiction. \end{proof}

\section{Different Graphs Giving the Same Matroid \texorpdfstring{$M_G$}{MG}} \label{section-non-isomorphic-graphs}

As mentioned after Definition~\ref{def:matroid_graphcurve}, 
we can have non-isomorphic graphs, with non-isomorphic cographic matroids, defining isomorphic graph curve matroids. This phenomenon is already mentioned in \cite{geiger2022selfdual}. Describing criteria for when two non-isomorphic graphs define the same graph curve matroid remains an open question. Here, we provide a sufficient condition using \emph{2-switches} of graphs, as defined in \cite{switching}.

Let $G_1 = (V_1, E_1)$ and $G_2 = (V_2, E_2)$ be two $2$-connected graphs. Let $G$ be any graph obtained from the disjoint union $G_1 \sqcup G_2$ via $2$-switching any two edges $(a_1, b_1) \in E_1$ 
and $(a_2,b_2) \in E_2$, i.e., replacing them with two new edges $(a_1,a_2)$ and $(b_1,b_2)$, as pictured in Figure~\ref{fig:graph-direct-sum}. 

\begin{figure}[h!]
    \centering

    \begin{tikzpicture}[scale = 0.27]
        \coordinate (18) at (0,0);
        \coordinate (1) at (-2,3);
        \coordinate (2) at (2,3);
        \coordinate (3) at (0,6);
        
        \foreach \i in {1,2,3,18} {\path (\i) node[circle, black, fill, inner sep=2]{};}

        \draw[black] (18)--(1)--(2)--(18);
        \draw[black] (1)--(3)--(2);
        \draw[black] (3) .. controls (5,6) and (5,0) .. (18);

        \node[below] at (18) {\small $b_1$};
        \node[above] at (3) {\small $a_1$};

        \coordinate (4) at (11,0);
        \coordinate (5) at (9,3);
        \coordinate (6) at (13,3);
        \coordinate (7) at (11,6);

        \foreach \i in {4,5,6,7} {\path (\i) node[circle, black, fill, inner sep=2]{};}

        \draw[black] (4)--(5)--(6)--(4);
        \draw[black] (5)--(7)--(6);
        \draw[black] (7) .. controls (6,6) and (6,0) .. (4);

        \node[below] at (4) {\small $b_2$};
        \node[above] at (7) {\small $a_2$};

        \node[right] at (17,3) {\large $\rightsquigarrow$};

        \coordinate (8) at (26,0);
        \coordinate (9) at (24,3);
        \coordinate (10) at (28,3);
        \coordinate (11) at (26,6);

        \coordinate (12) at (33,0);
        \coordinate (13) at (31,3);
        \coordinate (14) at (35,3);
        \coordinate (15) at (33,6);

        \foreach \i in {8,9,10,11,12,13,14,15} {\path (\i) node[circle, black, fill, inner sep=2]{};}

        \draw[black] (8)--(9)--(10)--(8);
        \draw[black] (9)--(11)--(10);
        \draw[black] (12)--(13)--(14)--(12);
        \draw[black] (13)--(15)--(14);
        \draw[black] (8)--(12);
        \draw[black] (11)--(15);

        \node[below] at (8) {\small $b_1$};
        \node[above] at (11) {\small $a_1$};
        \node[below] at (12) {\small $b_2$};
        \node[above] at (15) {\small $a_2$};
    \end{tikzpicture}
    
    \caption{A $2$-switch of the disjoint union of two graphs.}
    \label{fig:graph-direct-sum}
\end{figure}
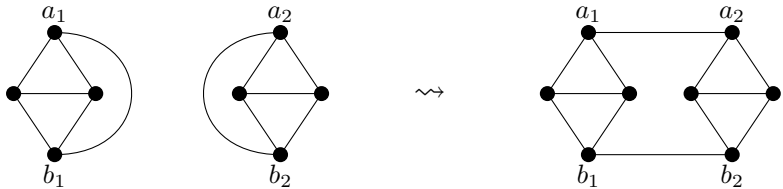

Using the first statement of Lemma \ref{lemma_rank_bondmatroid} (which does not require the running trivalence assumption from Section~\ref{section_circuits}) and Proposition \ref{prop_dependenceM_G}, it follows easily that $A_1 \subseteq V_1$ and $A_2 \subseteq V_2$ are independent in $M_{G_1}$ and $M_{G_2}$, respectively, if and only if $A_1 \cup A_2$ is independent in  $M_{G}$. We therefore obtain the following result.

\begin{proposition}\label{prop:direct-sum-split}
In the setting above, we have that $M_G = M_{G_1} \oplus M_{G_2}$.
\end{proposition}
In particular, this proves that the set of graph curve matroids coming from $2$-connected graphs is closed under taking direct sums.

Moreover, any $2$-connected graph $G$ that admits an edge cut of size $2$ can be written as the $2$-switch of the disjoint union of two smaller $2$-connected graphs. Thus, its associated matroid $M_G$ splits as a direct sum by Proposition~\ref{prop:direct-sum-split}. Therefore, the graph curve matroids arising from $3$-connected graphs are the building blocks of the class of graph curve matroids associated to $2$-connected graphs.

\begin{example}
The theta graph $G$ in Figure~\ref{fig:theta} has $M_G = U_{1,2}$. The soda can graph $H$ from Figure~\ref{fig:soda_can} is a $2$-switch of the disjoint union of two theta graphs, and we saw in Example~\ref{ex_soda} that $M_H =  U_{1,2} \oplus U_{1,2}$.
\end{example}

\begin{corollary}
    Two non-isomorphic graphs define the same graph curve matroid if they are both obtained from the same smaller $2$-connected graphs via $2$-switching different 
    choices of edges.
\end{corollary}

\bibliographystyle{amsalpha}
\bibliography{biblio}{}

\end{document}